\documentclass{amsart}
\usepackage{graphicx}
\usepackage{amsfonts}
\newtheorem{tm}{Theorem}

\newtheorem{defi}[tm]{Definition}

\newtheorem{rem}[tm]{Remark}

\newtheorem{ex}[tm]{Example}

\newtheorem{prop}[tm]{Proposition}
\newtheorem{nota}[tm]{Notation}
\newtheorem{quest}[tm]{Question}

\begin{document}

\title{Descartes' rule of signs, canonical sign patterns and rigid orders of
  moduli}
\author{Vladimir Petrov Kostov}
\address{Universit\'e C\^ote d’Azur, CNRS, LJAD, France} 
\email{vladimir.kostov@unice.fr}

\begin{abstract}
  We consider real
  polynomials in one variable without vanishing coefficients and with
  all roots real and of distinct moduli.
  We show that the signs of the coefficients define the order of
  the moduli of the roots on the real positive half-line exactly when  
no four
  consecutive signs of coefficients equal $(+,+,-,-)$, $(-,-,+,+)$, 
  $(+,-,-,+)$ or $(-,+,+,-)$.


\end{abstract}
\maketitle

\section{Introduction}

We consider real univariate polynomials. When all roots of such a polynomial
are real, the polynomial is called {\em hyperbolic}. We are interested in
the {\em generic} case, i.~e. the case of
hyperbolic polynomials without multiple roots and with all coefficients
non-vanishing (hence without root at~$0$).

The present paper continues a recent activity in the research about
real univariate polynomials and their roots. The classical Descartes'
rule of signs applied to a real degree $d$
polynomial $Q$ without vanishing coefficients
states that the number $pos$ of its positive roots is bounded by the
number of sign changes $c$ in the sequence $\mathcal{A}$
of its coefficients the difference
$c-pos$ being even. The same rule applied to the polynomial $Q(-x)$ implies
that the number $neg$ of negative roots of $Q$ is not larger than the
number $p$ of sign preservations in~$\mathcal{A}$, $c+p=d$,
and the difference $p-neg$
is also even. For a hyperbolic polynomial one has $pos=c$ and $neg=p$. 

The problem for which  couples $(pos, neg)$ satisfying these
conditions can one
find a real polynomial with prescribed signs of its coefficients
seems to have been explicitly formulated for the first time in~\cite{AJS}.
The first non-trivial result occurs for $d=4$, see~\cite{Gr}. The exhaustive
answer to the question for $d=5$ and $6$ is provided in \cite{AlFu}; for
$d=7$ and $8$ it is given in  \cite{FoKoSh} and \cite{KoCzMJ}. For
$4\leq d\leq 8$, in all compatible but not realizable cases one has 
$pos=0$ or $neg=0$. For $d\geq 9$, there are compatible and non-realizable
couples with $pos\geq 1$ and $neg\geq 1$, see  \cite{KoMB} and~\cite{CGK}.
For a tropical analog of  Descartes' rule of signs a discussion can be found
in~\cite{FoNoSh}. Various problems concerning hyperbolic polynomials in one
variable are exposed in~\cite{Ko}.

\begin{defi}
  {\rm A {\em sign pattern} of length $d+1$ is a sequence
    of $d+1$ signs plus and/or minus. The polynomial
    $Q:=\sum _{j=0}^da_jx^j$, $a_j\in \mathbb{R}^*$,
    is said to define the sign pattern $\sigma (Q):=({\rm sgn}(a_d)$,
    ${\rm sgn}(a_{d-1})$, $\ldots$, ${\rm sgn}(a_0))$. Most often we set
    $a_d:=1$. In this case the first sign of $\sigma (Q)$ is a plus.}
\end{defi}

A natural question to ask about a degree $d$ generic hyperbolic polynomial
is: Suppose that the
moduli of all its $d$ roots are distinct.  
When these moduli are arranged on the real positive
half-axis, at which positions can the moduli of its negative roots be? This
question leads to the following formal definition:

\begin{defi}
{\rm An {\em order of moduli} of length $d$ is a string of $p$
  letters $N$ and $c$ letters $P$, $c+p=d$, separated by signs
  of inequality $<$.
    The letters $N$ and $P$ indicate the relative positions of the moduli
    of the negative and positive roots of the hyperbolic polynomial
    on $\mathbb{R}_+$. 
    For example, to say that the degree $7$ hyperbolic polynomial
    $Q$ defines (or realizes)
    the order of moduli}
    $$N<N<P<N<P<N<P$$
{\rm means that for the sign pattern $\sigma (Q)$, one has $c=3$, $p=4$,
  and that for
    the positive roots $\alpha _1<\alpha _2<\alpha _3$ and the negative roots
    $-\gamma _j$ of $Q$, one has
    $$\gamma _1<\gamma _2<\alpha _1<\gamma _3<\alpha _2<\gamma _4<\alpha _3~.$$
A given order of moduli {\em realizes} a given sign pattern if there
exists a hyperbolic polynomial which defines the given  order of moduli
and the given sign pattern.}
\end{defi}

For generic hyperbolic polynomials,
we give the exhaustive answer to the following question:

\begin{quest}\label{quest1}
  When does the sign pattern $\sigma (Q)$
  determine the order of moduli defined by the
  hyperbolic polynomial~$Q$?
\end{quest}

The answer is given in terms of the following definition:

\begin{defi}\label{deficanon}
  {\rm We define after each sign pattern of length $d+1$ its corresponding
    {\em canonical} order of moduli: the sign pattern is read from right to
    left and to each
    couple of opposite (resp. identical) consecutive signs one puts
    in correspondence the letter $P$
    (resp. $N$). Example: for $d=8$, to the sign pattern $(+,+,-,-,+,-,+,+,+,-)$
    corresponds the canonical order of moduli $P<N<N<P<P<P<N<P<N$.}
\end{defi}

\begin{rem}
  {\rm Each sign pattern is realizable by its canonical order of moduli, see
    Proposition~1 and its proof in~\cite{KoSe}. A sign pattern
    is called {\em canonical}
    if it is realizable only by its canonical order of moduli. So to 
    answer Question~\ref{quest1} means to say which sign patterns are canonical.}
\end{rem}

\begin{defi}\label{deficonfig}
  {\rm We call {\em configuration} any four consecutive components of a
    given sign pattern. For instance, the sign pattern
    from Definition~\ref{deficanon} contains $7$
    configurations the first, second and last of which are
    $(+,+,-,-)$, $(+,-,-,+)$ and $(+,+,+,-)$. We give names to the following
    four configurations:}
  $$A:=(+,+,-,-)~\, ,~\, B:=(-,-,+,+)~\, ,~\, C:=(+,-,-,+)~\, ,~\,
  D:=(-,+,+,-)~.$$
  {\rm It is clear that if the sign pattern $\sigma (Q(x))$
    contains the configuration
    $A$ or $B$ (resp. $C$ or $D$), then in the same positions the sign pattern
    $\sigma (Q(-x))$ contains
    the configuration $C$ or $D$ (resp. $A$ or~$B$).}
\end{defi}

Our main result is the following theorem:

\begin{tm}\label{tmmain}
  A sign pattern is canonical if and only if it contains none of
  configurations $A$~--~$D$.
  \end{tm}

\begin{rem}
  {\rm (1) The ``only if'' part of the theorem is proved in~\cite{KoSe} (see
    Theorem~2 therein),
    so we prove only the ``if'' part. The formulation of the result of
    \cite{KoSe} is not given in terms of the configurations $A$~--~$D$, but
    is equivalent to such a formulation. In~\cite{KoPuMaDe} it is shown that
    the following sign patterns with one or two sign changes are canonical:}

  $$\begin{array}{lll}
    (+,+,\ldots ,+,-)~,&&(+,-,-,\ldots ,-)~,\\ \\
    (+,+,\ldots ,+,-,+,+,\ldots ,+)&{\rm and}&
    (+,-,-,\ldots ,-,+)~.\end{array}$$

  {\rm (2) Each of configurations $A$ and $B$ contains an isolated sign
    change, i.e.
  a sign change between two sign
  preservations, and vice versa for configurations $C$ and~$D$. Thus one can
  reformulate Theorem~\ref{tmmain} as follows:}
  A sign pattern is canonical if and only if it contains no isolated
    sign change and no isolated sign preservation.
\end{rem}

We prove Theorem~\ref{tmmain} in Section~\ref{secprtmmain}. In
Section~\ref{seccomments} we discuss problems related to
Question~\ref{quest1}.

\section{Some related problems\protect\label{seccomments}}

It is obvious that one can ask the inverse to Question~\ref{quest1}:

\begin{quest}\label{quest2}
  When does an order of moduli determine the sign pattern?
  In other words, do there
  exist orders of moduli such that
  each of them realizes a single sign pattern? (We call such orders of moduli
  {\em rigid}.)
\end{quest}

The exhaustive answer to this question (given in~\cite{KoAr},
assuming that the leading coefficient of the generic hyperbolic polynomial
is positive) looks like this:

{\em These are the orders of moduli of the form $\cdots <P<N<P<N<\cdots$,
  and the
  corresponding sign patterns equal}

$$(+,+,-,-,+,+,-,-,\ldots )~~~\, {\it or}~~~\, (+,-,-,+,+,-,-,+,\ldots )~,$$
{\em and also the orders of moduli
  $\cdots <N<N<N<\cdots$ and $\cdots <P<P<P<\cdots$ the
  corresponding sign patterns being} 
$$(+,+,+,+,+,\ldots )~~~\, {\it and}~~~\, (+,-,+,-,+,\ldots )~.$$

The latter two sign patterns are trivially canonical, because all roots
are of the same
sign and hence one cannot compare moduli of roots of opposite signs.
For the former two sign patterns one
observes that every configuration of theirs is of the form $A$, $B$, $C$
or $D$.

Thus for any degree $d\geq 3$, there exist exactly two non-trivial
rigid orders of moduli and
corresponding sign patterns.
The number of canonical sign patterns, on the contrary,
increases with $d$ and tends to
infinity. To see this it suffices to notice that canonical are, in particular,
all sign patterns of the
form $([s_1],-,[s_2],-,[s_3],-,\ldots )$, where $[s_j]$ denotes a sequence of
$s_j$ consecutive signs plus, $s_j\geq 3$.

\section{Proof of Theorem~\protect\ref{tmmain}\protect\label{secprtmmain}}

In the proof of Theorem~\ref{tmmain} we use also
{\em generalized orders of moduli},
i.~e. orders of moduli in which one of the
signs of inequality~$<$ in a given order of moduli $\mathcal{S}$
is replaced by the sign of equality~$=$ with the
obvious meaning that the moduli of the corresponding two roots are equal.
We say that the thus obtained generalized order of moduli is {\em adjacent} to
the order of moduli~$\mathcal{S}$.

Consider a sign pattern $\sigma ^{\flat}$ of length $d+1$ and containing none of
configurations $A$~--~$D$. Then there exists a degree $d$ 
hyperbolic polynomial $U$ realizing
the sign pattern $\sigma ^{\flat}$ with the canonical order of
moduli $\mathcal{R}$
defined by $\sigma ^{\flat}$ 
(see \cite[Proposition~1]{KoSe}). Denote by $\Omega (\sigma ^{\flat})$ the
set of degree $d$ monic generic
hyperbolic polynomials whose coefficients define the
sign pattern $\sigma ^{\flat}$.
Hence the set $\Omega (\sigma ^{\flat})$ is open and contractible
(see \cite[Theorem~2]{KoCo}).

Suppose that the set $\Omega (\sigma ^{\flat})$ contains a polynomial $V$ which
defines an order of moduli different from $\mathcal{R}$. Then one can
connect the polynomials
$U$ and $V$ by a continuous path $\gamma \subset \Omega (\sigma ^{\flat})$.
The set $\Omega (\sigma ^{\flat})$ being open, one can choose $\gamma$ such that
for each point of it, in the corresponding polynomial there is at most one
equality between the moduli of a negative and of a positive root. Hence
there exists a polynomial $W^{\dagger}\in \gamma$ defining a generalized
order of moduli
adjacent to $\mathcal{R}$. Hence $W^{\dagger}$ is of the form
$(x^2-\alpha ^2)W_{\diamond}$,
where $\alpha >0$ and $W_{\diamond}$ is a degree $d-2$ monic
hyperbolic polynomial.
One can make the change
of variables $x\mapsto \alpha x$ after which the polynomial
$W:=W^{\dagger}/\alpha ^2$
takes the form

\begin{equation}\label{eqpolyW}
  W=(x^2-1)W_*~,~~~\,
  W_*:=x^{d-2}+u_1x^{d-3}+\cdots +u_{d-3}x+u_{d-2}~,~~~\, u_j\in \mathbb{R}~.
\end{equation}
Notice that the coefficients $u_j$ are not indexed in the same way 
as the coefficients $a_j$ above.

We want to prove that if the sign pattern $\sigma ^{\flat}$ contains none of
configurations $A$~--~$D$, then such a polynomial $W$ does not exist. Hence $V$
also does not exist which proves Theorem~\ref{tmmain}.

\begin{rem}\label{remconv}
  {\rm One can make two non-restrictive assumptions on $W$:
\vspace{1mm}

1) One can assume that $u_1>0$. Indeed, as
$W=x^d+u_1x^{d-1}+\cdots$, one has $u_1\neq 0$. If $u_1<0$, then
one can make the change $x\mapsto -x$ which preserves the factor $x^2-1$.
This means that to prove Theorem~\ref{tmmain} it suffices to consider only half
of the possible cases, the ones of sign patterns beginning with two pluses.
\vspace{1mm}

2) One can assume that $u_j\neq 0$, $2\leq j\leq d-2$. Indeed, if some
of the coefficients $u_j$ equals $0$, then for its nearby values the
coefficients $a_j$ retain the same signs, so the sign pattern of $W$
remains the same.}
\end{rem}

\begin{nota}\label{notaST}
  {\rm For a given sign pattern $\sigma ^{\triangle}$, we denote by
    $c(\sigma ^{\triangle})$ the number of its sign changes.
    For each sign pattern $\sigma (W_*)$ of length $d-1$ 
beginning with two pluses, we denote by 
$S_d(\sigma (W_*))$ the set of possible sign patterns $\sigma (W)$.
By $T_d(\sigma (W_*))\subseteq S_d(\sigma (W_*))$ we denote its subset
of sign patterns $\sigma (W)$ such that $c(\sigma (W))=c(\sigma (W_*))+1$.}
  \end{nota}

\begin{ex}\label{explus}
  {\rm Suppose that $d=5$ and $\sigma (W_*)=(+,+,+,+)$. The coefficients of
    $W$ equal}
  $$1~,~~~\, u_1~,~~~\, u_2-1~,~~~\, u_3-u_1~,~~~\, -u_2~,~~~\, -u_3~.$$
  {\rm One has $u_1>0$, $-u_2<0$ and $-u_3<0$. A priori the
    coefficients $u_2-1$ and $u_3-u_1$ can have any sign. Therefore the set
    $S_5(\sigma (W_*))$ consists of the four sign patterns of the form
    $(+,+,\pm ,\pm ,-,-)$. As $c(\sigma (W_*))=0$, for the sign patterns
    of $T_5(\sigma (W_*))$ one has $c(\sigma (W))=1$, so} 

  $$T_5(\sigma (W_*))=\{ (+,+,+,+,-,-)~,~(+,+,+,-,-,-)~,~(+,+,-,-,-,-)\} $$
  {\rm and $S_5(\sigma (W_*))\setminus T_5(\sigma (W_*))=\{ (+,+,+,-,+,-)\}$.
    We consider only polynomials $W$ with all coefficients non-vanishing,
    therefore the possibilities to have $u_2-1=0$ and/or $u_3-u_1=0$ are
    not taken into account.

  More generally, suppose that $d\geq 3$.
  If $u_j>0$, $1\leq j\leq d-2$ (see (\ref{eqpolyW}), then
  $c(\sigma (W_*))=0$. 
    The first two and the last two coefficients of $W$ equal
    $1>0$, $u_1>0$ and $-u_{d-3}<0$, $-u_{d-2}<0$ respectively. 
    If $\sigma (W)\in T_d(\sigma (W_*))$ (so $c(\sigma (W))=1$),
    then $\sigma (W)$ consists of $m\geq 2$ signs plus
    followed by $d+1-m\geq 2$ signs minus.
    Hence this sign pattern contains a
    configuration~$A$.}
\end{ex}

\begin{rem}
  {\rm The sets $S_d(\sigma (W_*))$ and $T_d(\sigma (W_*))$ can be defined for
    any real monic univariate, not necessarily hyperbolic, polynomial $W_*$.
    In this case the polynomial $W$ is also not necessarily hyperbolic. When
    the polynomials $W_*$ and $W$ are hyperbolic, they have $c(\sigma (W_*))$
    and $c(\sigma (W_*))+1$ positive and $d-2-c(\sigma (W_*))$
    and $d-1-c(\sigma (W_*))$ negative roots respectively,
    see the second paragraph of the Introduction.}
\end{rem}

Theorem~\ref{tmmain} results from the following proposition:

\begin{prop}\label{proptmmain}
  For any real monic univariate polynomial $W_*$ as in (\ref{eqpolyW}),
  every sign pattern of its corresponding set $T_d(\sigma (W_*))$ contains
  at least one of configurations $A$~--~$D$.
  \end{prop}
    
Hence if the sign pattern of the generic hyperbolic polynomial $W$ contains
none of configurations $A$~--~$D$, then $W$ is not representable in the
form~(\ref{eqpolyW}).

We give first examples of polynomials $W$ in some of which we use a different
notation. These examples will be used in the proof
of Proposition~\ref{proptmmain}.


\begin{ex}\label{exd3}
  {\rm Set $d:=3$. Consider the polynomial $(x^2-1)(x+a)=x^3+ax^2-x-a$,
    $a\in \mathbb{R}^*$. For $a>0$ (resp. for $a<0$), it defines the
    sign pattern
    $(+,+,-,-)$ (resp. $(+,-,-,+)$) which is configuration~$A$ (resp.
    configuration~$C$).}
\end{ex}

\begin{ex}\label{exd4}
  {\rm Set $d:=4$. Consider for $a$, $b\in \mathbb{R}^*$, $a>0$, the polynomial}

  $$(x^2-1)(x^2+ax+b)=x^4+ax^3+(b-1)x^2-ax-b~.$$
  {\rm For $b>0$, it defines one of the sign patterns $(+,+,\pm ,-,-)$
    both of which contain configuration $A$.
    For $b<0$, it defines the sign pattern
    $(+,+,-,-,+)$ containing configurations $A$ and~$C$.}
\end{ex}

\begin{ex}\label{exd5}
  {\rm Set $d:=5$. Consider for $a$, $b$, $r\in \mathbb{R}^*$, $a>0$,
    the polynomial}

  $$W:=(x^2-1)(x^3+ax^2+bx+r)=x^5+ax^4+(b-1)x^3+(r-a)x^2-bx-r~.$$
  {\rm For $b>0$ and $r<0$, it defines one of the
    sign patterns $(+,+,\pm ,-,-,+)$; 
    for $b<0$ and $r>0$, it defines one of the sign patterns
    $(+,+,-,\pm ,+,-)$; for $b<0$ and $r<0$, it defines the sign pattern
    $(+,+,-,-,+,+)$. Each of these sign patterns contains at least one of
    configurations $A$~--~$D$.}

  {\rm For $b>0$ and $r>0$, one obtains the sign patterns
    $(+,+,\pm ,\pm ,-,-)$ of which
    only $(+,+,-,+,-,-)$ contains neither of
    configurations $A$~--~$D$. However this sign pattern has three
    sign changes, so it does not belong to the set $T_5(\sigma (W_*))$
    with $\sigma (W_*)=(+,+,+,+)$, see Example~\ref{explus}.}
  \end{ex}

\begin{ex}\label{exd6}
  {\rm Set $d:=6$.
    Consider for $a$, $b$, $r$, $g\in \mathbb{R}^*$, $a>0$, the polynomial}

  $$(x^2-1)(x^4+ax^3+bx^2+rx+g)=x^6+ax^5+(b-1)x^4+(r-a)x^3+(g-b)x^2-rx-g~.$$
  {\rm We give a list of cases in each of which either the
    sign pattern $c(\sigma (W))$ thus obtained contains at least
    one of configurations $A$~--~$D$ or the condition
    $c(\sigma (W))=c(\sigma (W_*))+1$ is violated (this takes place
    in the sign pattern $(+,+,-,+,-,-,+)$ of case 7)):}
  \vspace{1mm}
  
  {\rm 1) $b<0$, $r>0$, $g>0$: $(+,+,-,\pm ,+,-,-)$;}
  \vspace{1mm}
  
  {\rm 2) $b<0$, $r>0$, $g<0$: $(+,+,-,\pm ,\pm ,-,+)$;}
  \vspace{1mm}
  
  {\rm 3) $b<0$, $r<0$, $g>0$: $(+,+,-,-,+,+,-)$;}
  \vspace{1mm}
  
  {\rm 4) $b<0$, $r<0$, $g<0$: $(+,+,-,-,\pm ,+,+)$;}
  \vspace{1mm}
  
  {\rm 5) $b>0$, $r<0$, $g>0$: $(+,+,\pm ,-,\pm ,+,-)$;}
  \vspace{1mm}
  
  {\rm 6) $b>0$, $r<0$, $g<0$: $(+,+,\pm ,-,-,+,+)$;}
  \vspace{1mm}
  
  {\rm 7) $b>0$, $r>0$, $g<0$: $(+,+,\pm ,\pm ,-,-,+)$.}
  \vspace{1mm}
  
  {\rm The only case not included in this list is  
$b>0$, $r>0$, $g>0$. It is covered by Example~\ref{explus}.}
\end{ex}



\begin{proof}[Proof of Proposition~\ref{proptmmain}]
  We prove Proposition~\ref{proptmmain} by induction on~$d$. The induction base
  are the cases $3\leq d\leq 6$ considered in
  Examples~\ref{exd3}~--~\ref{exd6}.

  We compare the coefficients of the polynomial $W$ for two consecutive degrees,
  $d$ and $d+1$, $d\geq 6$. 
  We denote these polynomials by $W_d$ and $W_{d+1}$ and their corresponding
  polynomials $W_*$ (see~(\ref{eqpolyW})) by $W_{d,*}$ and $W_{d+1,*}$.  
  The coefficients of $W_d$ and $W_{d+1}$ are:

  \begin{equation}\label{eqcompare}
    \begin{array}{rrrrcrrrr}
      1&u_1&u_2-1&\ldots &u_{d-3}-u_{d-5}&u_{d-2}-u_{d-4}&-u_{d-3}&-u_{d-2}&\\ \\
      1&u_1&u_2-1&\ldots &u_{d-3}-u_{d-5}&u_{d-2}-u_{d-4}&u_{d-1}-u_{d-3}&-u_{d-2}
      &-u_{d-1}~.\end{array}
  \end{equation}

\begin{defi}\label{defiM}
  {\rm We say that a sign pattern $\sigma ^{\bullet}$
    of length $d+1$ belongs to the class
      $M_d(m)$, $1\leq m\leq d-2$, if it contains one of configurations
      $A$~--~$D$ in positions $m$, $m+1$, $m+2$, $m+3$.
    In particular, if $\sigma ^{\bullet}\in M_d(d-4)$, then the following
    coefficients of $\sigma ^{\bullet}$ form one of configurations $A$~--$D$:}

  $$ u_{d-5}-u_{d-7}~,~~~\, u_{d-4}-u_{d-6}~,~~~\, u_{d-3}-u_{d-5}~~~\,
  {\rm and}~~~\, u_{d-2}-u_{d-4}~.$$
  {\rm The coefficient $u_{d-2}-u_{d-4}$ is the rightmost of the coefficients
  which are the same for $W_d$ and $W_{d+1}$, see~(\ref{eqcompare}).}
  \end{defi}

  \begin{rem}\label{remcompare}
    {\rm When passing from $W_d$ to $W_{d+1}$, one deduces from
      (\ref{eqcompare}) that:
      \vspace{1mm}
      
  {\em i)} Changes occur
  only among the coefficients at the right end.
  Namely, the coefficient $-u_{d-3}$ becomes $u_{d-1}-u_{d-3}$ and the last
  coefficient $-u_{d-1}$ is added.
  \vspace{1mm}
  
  
  {\em ii)} If the sign pattern $\sigma (W_d)$ is in the class
  $M_d(m_0)$ for some
  $1\leq m_0\leq d-4$, then $\sigma (W_{d+1})$ is in the class $M_{d+1}(m_0)$,
  see Definition~\ref{defiM}.}
    \end{rem}

  By induction hypothesis the sign pattern $\sigma (W_d)$ belongs to
  at least one of the classes $M_d(m)$. If $m\leq d-4$, by {\em ii)}
  of Remark~\ref{remcompare}, one has $\sigma (W_{d+1})\in M_{d+1}(m)$.
  Therefore we need to
  consider only the cases

  $${\rm I)}~\sigma (W_d)\in M_d(d-2)~~~\, \, {\rm and}~~~\, \, 
    {\rm II)}~M_d(d-2)\not\ni\sigma (W_d)\in M_d(d-3)~.$$ 

      {\em Case} I). Suppose first that sgn$(u_{d-1}-u_{d-3})=$sgn$(-u_{d-3})$. 
        This is true when, but not only when sgn$(u_{d-1})=-$sgn$(u_{d-3})$. 
        Hence all signs of the sign pattern $\sigma (W_{d+1})$ except
          the last one coincide with the corresponding signs
          of the sign pattern $\sigma (W_d)$, so 
            $\sigma (W_{d+1})\in M_{d+1}(d-2)$.

          Suppose that sgn$(u_{d-1}-u_{d-3})=-$sgn$(-u_{d-3})$.
          Hence sgn$(u_{d-1})=$sgn$(u_{d-3})$.
   We list the possible
   last four signs of $\sigma (W_d)$ in the first line and the
     corresponding last five signs of $\sigma (W_{d+1})$ in the second line:

       $$\begin{array}{llll}
         (+,+,-,-)&(-,-,+,+)&(+,-,-,+)&(-,+,+,-)\\ \\
         (+,+,+,-,-)&(-,-,-,+,+)&(+,-,+,+,-)&(-,+,-,-,+)~.\end{array}$$
       In all four cases the sign pattern $\sigma (W_{d+1})$ contains one of
         configurations $A$~--~$D$ in its last four positions.

         {\em Case} II). We list
         the last five signs of $\sigma (W_d)$ in the first line and the
         corresponding last six signs of $\sigma (W_{d+1})$ in the second line:

         $$\begin{array}{llll}
         (+,+,-,-,-)&(-,-,+,+,+)&(+,-,-,+,-)&(-,+,+,-,+)\\ \\
         (+,+,-,+,-,-)&(-,-,+,-,+,+)&(+,-,-,-,-,+)&(-,+,+,+,+,-)~.\end{array}$$
         In the first two cases one has sgn$(u_{d-1})=$sgn$(u_{d-2})$, so
         $c(W_{d,*})=c(W_{d+1,*})$. Therefore one should have

         $$c(W_{d+1,*})+1=c(W_{d+1})=c(W_d)=c(W_{d,*})+1~.$$ 
         However one has $c(W_{d+1})=c(W_d)+2$.

         In the last two cases
         the equality sgn$(u_{d-1})=-$sgn$(u_{d-2})$ implies 
         $c(W_{d+1,*})=c(W_{d,*})+1$, so one should
         have $c(W_{d+1})=c(W_d)+1$, but one has $c(W_{d+1})=c(W_d)-1$.
         This contradiction proves the proposition.
\end{proof}

\end{document}